\documentclass{birkjour}
\usepackage[colorlinks]{hyperref}

 \newtheorem{thm}{Theorem}[section]
 \newtheorem{cor}[thm]{Corollary}
 \newtheorem{lem}[thm]{Lemma}
 \newtheorem{prop}[thm]{Proposition}
\theoremstyle{definition}
\newtheorem{defn}[thm]{Definition}
\theoremstyle{remark}
 \newtheorem{rem}[thm]{Remark}
 
 \numberwithin{equation}{section}

\newcommand{\di}{\mathcal{D}}
\newcommand{\bdi}{\widetilde{\mathcal{D}}}
\newcommand{\co}{\nabla}
\newcommand{\bco}{\widetilde{\nabla}}
\newcommand{\g}{\widetilde{g}}
\newcommand{\n}{{N_1 }_{f_2}\!\!\times_{f_1}N_2}
\newcommand{\m}{{M_1 }_{\rho_2}\!\!\times_{\rho_1}M_2}

\newcommand{\grad}{\text{grad}}

\begin{document}

%
%
%
%
%
%
%
%
%

\title[On doubly warped product immersions]
 {On doubly warped product immersions}

\author[M. Faghfouri ]{Morteza Faghfouri }

\address{%
Faculty of mathematics \\
University of Tabriz\\
Tabriz, Iran}

\email{faghfouri@tabrizu.ac.ir}

\author[A. Majidi]{Ayyoub Majidi}
\address{%
Faculty of mathematics \\
University of Tabriz\\
Tabriz, Iran}
\email{a\_majidi89@ms.tabrizu.ac.ir}



\begin{abstract}
In this paper we study fundamental geometric properties of doubly warped product immersion which is an  extension of warped product immersion. Moreover,  we study  geometric  inequality for doubly warped products
isometrically immersed in arbitrary Riemannian manifolds.
\end{abstract}

\maketitle

\section{Introduction}
Warped products were first defined by  Bishop and O'Niell
in \cite{bishop.oneill:}. O'Niell  discussed warped products and explored curvature formulas of warped products in terms of curvatures of  components of warped products in \cite{oneil:book}.
B. Y. Chen  \cite{chen:onwarpedimmersions,chen:on.isometric.minimal.immersions} studied the fundamental geometry properties of warped product immersions. In general, doubly warped products can be considered as a generalization of singly warped products\cite{unal:doubly.warped.products.th,unal:doubly.warped.products}. A. Olteanu \cite{olteanu:generalinequalityfordoublywarped},  S.~Sular and C.~{\"O}zg{\"u}r \cite{sular.ozgur:Doublywarpedproduct} and K.~Matsumoto in \cite{matsumoto:doublywarpedsubmanifold}  extended  some properties of warped product, submanifolds  and geometric inequality in warped product manifolds  for doubly warped product submanifolds into arbitrary Riemannian manifolds.  In this paper,  we extend  some properties of warped product immersion obtained in \cite{chen:onwarpedimmersions}, geometric inequality and minimal immersion problem studied in \cite{chen:on.isometric.minimal.immersions,olteanu:generalinequalityfordoublywarped} to doubly warped product.

\section{Preliminary }
Let $N_1$ and $N_2$ be  two Riemannian manifolds equipped with Riemannian metrics $g_1$ and $g_2$ , respectively, and let $f_1$ and $f_2$  be a positive differentiable functions on $N_1$ and $N_2$, respectively.  The doubly warped product ${N_1 }_{f_2}\!\!\times_{f_1}N_2$ is defined to be the product manifold $N_1\times N_2$ equipped with the Riemannian metric given by $$g=(f_2o\pi_2)^2\pi_1^*(g_1)+(f_1o\pi_1)^2\pi_2^*(g_2)$$ where $\pi_i:N_1\times N_2 \to N_i$ is the natural projections. We denote the dimension of $N_1$ and $N_2$ by $n_1$ and $n_2$, respectively. In particular, if $f_2=1$ then ${N_1 }_{1}\!\!\times_{f_1}N_2={N_1 }\times_{f_1}N_2$ is called warped product of $(N_1,g_1)$ and $(N_2,g_2)$ with warping function $f_1$.

For a vector field $X$ on $N_1$, the lift of $X$  to $\n$ is the vector field $\tilde{X}$ whose value at each $(p,q)$ is the lift $X_p$ to $(p,q)$. Thus the lift of $X$ is the unique vector field on $\n$ that is $\pi_1$-related to $X$ and $\pi_2$-related to the zero vector field on $N_2$.
For a doubly warped product ${N_1 }_{f_2}\!\!\times_{f_1}N_2$, let $\di_i$ denotes the distribution obtained from the vectors tangent to the horizontal lifts of $N_i$.

\begin{lem}
Let $N={N_1 }_{f_2}\!\!\times_{f_1}N_2$  be a doubly warped product of Riemannian manifolds  $N_1 \text{ and } N_2$. If we put
$U_i=-\grad((\ln f_i)o\pi_i)$
then the Levi-Civita connection $\co$ and curvature tensor $R$ of the doubly warped product is related to the Levi-Civita connection $\co^0$ and curvature $R^0$ of the direct product of $(N_1,g_1)$ and $(N_2,g_2)$ equipped with the direct product metric $g_0=g_1+g_2$ by
\begin{align}
\co_XY&=\co^0_XY+\sum_{i=1}^{2}(\langle X^i,Y^i\rangle.U_i-\langle X,U_i\rangle.Y^i-\langle Y,U_i\rangle.X^i)\label{eq:doubly.co}\\
R(X,Y)&=R^0(X,Y)+\sum_{i=1}^{2}\{(\co_XU_i-\langle X,U_i\rangle U_i)\wedge Y^i\nonumber\\
&\qquad-(\co_YU_i-\langle Y,U_i\rangle U_i)\wedge X^i\}+\sum_{i,j=1}^{2}\langle U_i,U_j\rangle X^i\wedge Y^j\label{eq:doubly.cu}
\end{align}
where $X^i$ is the $N_i$-component of $X$ and $X\wedge Y$ is defined by
$$(X\wedge Y)Z=\langle Y,Z\rangle X-\langle X,Z\rangle Y.$$
\end{lem}

Let $\phi:N\to M$ be an isometric immersion of a Riemannian manifold $N$ into a Riemannian manifold $M$. The formulas of Gauss and Weingarten are given respectively by
\begin{align}
\widetilde{\co}_XY&=\co_XY+h(X,Y);\label{eq:gauss}\\
\bco_X\eta&=-A_\eta X+D_X\eta\label{eq:win}
\end{align}
for all vector fields $X,Y$ tangent to $N$ and $\eta$ normal to $N$, where $\bco$ denotes the Levi-Civita connection on $M$, $h$ the second fundamental form, $D$ the normal connection, and $A$ the shape operator of $\phi:N\to M$. The second fundamental form and the shape operator are related by $\langle A_\eta X,Y\rangle=\langle h(X,Y),\eta\rangle$, where $\langle,\rangle$ denotes the inner product on $M$.

The equation of Gauss of $\phi:N\to M$ is given by
\begin{align}
\langle \widetilde{R}(X,Y)Z,W\rangle=\langle R(X,Y)Z,W\rangle+\langle h(X,Z),h(Y,W)\rangle\nonumber\\
\qquad-\langle h(X,W),h(Y,Z)\rangle,\label{eq:gauss curvature}
\end{align}
for $X, Y, Z, W\in\Gamma(TN)$.

If a Riemannian manifold $M$ is of constant curvature
$c$, we have
 \begin{align}
\langle R(X,Y)Z,W\rangle&=c\{\langle Y,Z \rangle\langle X,W\rangle-\langle X,Z \rangle\langle Y,W\rangle\}\label{eq:curvature:cons}\\
&\qquad +\langle h(Y,Z), h(X,W)\rangle-\langle h(X,Z), h(Y,W)\rangle.\nonumber
\end{align}

The mean curvature vector $\textbf{H}$ is defined by
$\textbf{H}=\frac{1}{n}\text{trace }h$.
For a normal vector field $\eta$ on N, if $A_\eta=\lambda I$ for some $\lambda\in C^\infty(N)$, then
$\eta$ is called an umbilical section, or $N$ is said to be umbilical with respect
to $\eta$. If the submanifold $N$ is umbilical with respect to every local normal
vector field, then $N$ is called a totally umbilical submanifold.

A submanifold $N$ is called minimal in $M$ if the mean curvature vector of $N$ in $M$
vanishes identically. A submanifold in
a Riemannian manifold is called totally geodesic if its second fundamental form
vanishes identically.

Let $\phi:N\to M$ be an isometric immersion and $f\in C^\infty(M)$. We denote by $\co f$ and $D f$ the gradient of $f$ and the normal component of $\co f$ restricted on $N$, respectively.

Let $\psi$ a differential function on a Riemannian n-manifold $N$. Then  the Hessian tensor field of $\psi$ given by
\begin{align}\label{eq:hessian}
H^\psi(X,Y)=XY\psi-(\co_XY)\psi
\end{align}
and
the Laplacian of $\psi$ is given by
\begin{align}\label{eq:laplac}
\Delta\psi=-\text{trace }(H^\psi)=\sum_{i=1}^{n}((\co_{e_i}e_i)\psi-e_ie_i\psi),
\end{align}
where $e_1,\ldots,e_n$ is an orthonormal frame field on $N$. If $\Delta\psi=0$, then $\psi$ is called harmonic.

 We denote  by $K(X\wedge Y)$ the sectional curvature of the plane section spanned by $X,Y$.
\begin{lem}[\cite{chen:2011pseudo}]\label{lemma:mini}
Every  minimal submanifold $N$ in a Euclidean space $\mathbb{E}^m$
is non-compact.
\end{lem}
\begin{lem}[\cite{chen:2011pseudo}]\label{lemma:hopf1}
Every harmonic function on a compact Riemannian manifold
is a constant.
\end{lem}
\begin{lem}[Hopf's lemma. in \cite{chen:2011pseudo} ]\label{lemma:hopf2}
Let $M$ be a compact Riemannian $n$-manifold. If $\psi$ is a
differentiable function on $M$ such that $\Delta\psi \geq0$  everywhere on $M$ (or $\Delta\psi \geq0$
everywhere on $M$), then $\psi$ is a constant function.
\end{lem}

\section{Isometric immersion of a doubly warped product}

Let $\phi:{N_1 }_{f_2}\!\!\times_{f_1}N_2\to M$ be an isometric immersion of a doubly warped product  ${N_1 }_{f_2}\!\!\times_{f_1}N_2$ into a Riemannian manifold $M$. Let $h_i$ denote the restriction of the second fundamental form to $\di_i, i=1,2$. Denote by trace $h_i$ the trace of $h_i$ restricted to $N_i$, i.e.,
$$\text{trace }h_i=\sum_{\alpha=1}^{n_i}h(e_\alpha,e_\alpha)$$
for an orthonormal frame fields $e_1, \ldots,e_{n_i}$ of $\di_i$.

The partial mean curvature vector $H_i$ is defined by
\begin{align}
H_i=\dfrac{\text{trace } h_i }{n_i}, i=1,2
\end{align}
\begin{defn}
An immersion $\phi:{N_1 }_{f_2}\!\!\times_{f_1}N_2\to M$ is called $N_i-$totally geodesic(resp. $N_i-$minimal) if $h_i$ (resp. $H_i$) vanishes identically. And $\phi$ is called mixed totally geodesic if its second fundamental form $h$ satisfies $h(X,Y)=0$ for any $X\in\di_1$ and $Y\in\di_2$\cite{chen:2011pseudo}.
\end{defn}

\begin{defn}
Suppose that an immersion ${M_1 }_{\rho_2}\!\!\times_{\rho_1}M_2$ is a doubly warped product and $\phi_i:N_i\to M_i, i=1,2$, are isometric immersions between Riemannian manifolds. Define a positive smooth function $f_i$ on $N_i$ by $f_i=\rho_i o\phi_i$.
The map
\begin{align}
\phi:{N_1 }_{f_2}\!\!\times_{f_1}N_2\to {M_1 }_{\rho_2}\!\!\times_{\rho_1}M_2
\end{align}
given by $\phi(x_1,x_2)=(\phi_1(x_1),\phi_2(x_2))$ is an isometric immersion, which is called a doubly  warped product immersion
\end{defn}

\begin{thm}
Let $(N_i,g_i)$ and $(M_i,\g_i), i=1,2$ be Riemannian manifolds and $\phi:N_i\to M_i, i=1,2,$ be isometric immersions. Assume that
$$ \phi=(\phi_{1} , \phi_{2}) : \n \to \m $$
be a doubly warped product immersion between two doubly warped product manifolds. Then
\begin{enumerate}
\item\label{ca:1}
$ \phi $
is mixed totally geodesic.
\item\label{ca:2} The squared norm of the second fundamental form of
$ \phi $
satisfies
\begin{align}
\| h^\phi \|^{2}\geq n_{1} \| D\ln \rho_{2} \|^{2} + n_{2} \| D\ln \rho_{1} \|^{2},
\end{align}
where
$ n_{1}=dim N_{1} $
and
$ n_{2}=dim N_{2}. $
Equality holds if and only if
$ \phi_{1}:N_{1} \to  M_{1} $
and
$ \phi_{2}:N_{2} \to  M_{2} $
is both totally geodesic immersions.
\item\label{ca:3} $ \phi $ are
$ N_{i} $-totally geodesic if and only if
$ \phi_{i}:N_{i} \to  M_{i} $
are totally geodesic and
$ D\ln \rho_{j}=0(i \neq j ). $
\item\label{ca:4}
$ \phi $ is a totally geodesic immersion if and only if $ \phi $ is both $ N_{1} $- totally geodesic and $ N_{2} $- totally geodesic.
\end{enumerate}
\end{thm}
\begin{proof}
Denote by $\co^0$ and $\co$ the Levi-Civita connections of $N_1\times N_2$ equipped with the direct product metric and with doubly warped product metric, respectively. Similarly, denote by $\bco^0$ and $\bco$ the Levi-Civita connections of $M_1\times M_2$ equipped with the direct product metric and with doubly warped product metric, respectively.

From \eqref{eq:doubly.co} for vector fields $X,Y\in \bdi_1$ and $Z,W\in\bdi_2$ we obtain.
\begin{align}
\bco_XY&=\bco^0_XY-\g(X,Y)\grad(\ln\rho_2)\label{eq:doubly.co.m1}\\
\bco_ZW&=\bco^0_ZW-\g(Z,W)\grad(\ln\rho_1)\label{eq:doubly.co.m2}\\
\bco_XZ&=\bco^0_ZX=X(\ln\rho_1)Z+Z(\ln\rho_2)X\label{eq:doubly.co.m12}
\end{align}
Denoted by $h^\phi$ and $h^0$ the second fundamental form of a doubly warped product immersion $\phi=(\phi_{1} , \phi_{2}) : \n \to \m $ and the second fundamental form of the corresponding direct product immersion $\phi=(\phi_{1} , \phi_{2}) : N_1\times N_2 \to M_1\times M_2 $, respectively.
By applying \eqref{eq:doubly.co.m1}, \eqref{eq:doubly.co.m2} and \eqref{eq:doubly.co.m12}, we obtain
\begin{align}
h^\phi(X,Y)&=h^0(X,Y)-\g(X,Y)D(\ln\rho_2),\quad X,Y\in\di_1\label{eq:doubly.fun.m1}\\
h^\phi(Z,W)&=h^0(Z,W)-\g(Z,W)D(\ln\rho_1),\quad Z,W\in\di_2\label{eq:doubly.fun.m2}\\
h^\phi(X,Z)&=0\label{eq:doubly.fun.m12},\quad X\in\di_1, Z\in\di_2.
\end{align}
The restriction of $h^0$ to $\di_i, i=1,2$  are the second fundamental form of $\phi_i:N_i\to M_i, i=1,2$.  Hence, $h^0(X,Y)$ and $h^0(Z,W)$ are orthogonal in $X,Y\in\di_1$ and  $Z,W\in\di_2$.\\[2mm]
\ref{ca:1}. This follows from equation \eqref{eq:doubly.fun.m12}.\\[2mm]
\ref{ca:2}. Let $e_i\in\di_1, i=1,\ldots,n_1$ and $e_\alpha\in\di_2, \alpha=n_1+1,\ldots,n_1+n_2$  be orthonormal frame fields of $\n$.
\begin{align*}
||h^\phi||^2=&\sum_{a,b=1}^{n_1+n_2}\langle h^\phi(e_a,e_b),h^\phi(e_a,e_b)\rangle\\
=&\sum_{i,j=1}^{n_1}\langle h^\phi(e_i,e_j),h^\phi(e_i,e_j)\rangle+\sum_{\alpha,\beta=n_1+1}^{n_1+n_2}\langle h^\phi(e_\alpha,e_\beta),h^\phi(e_\alpha,e_\beta)\rangle\\
=&\sum_{i,j=1}^{n_1}\langle h^0(e_i,e_j)-\langle e_i,e_j\rangle D(\ln \rho_2),h^0(e_i,e_j)-\langle e_i,e_j\rangle D(\ln \rho_2)\rangle\\
&+\hspace{-4mm}\sum_{\alpha,\beta=n_1+1}^{n_1+n_2}\langle h^0(e_\alpha,e_\beta)-\langle e_\alpha,e_\beta\rangle D(\ln \rho_1),h^0(e_\alpha,e_\beta)-\langle e_\alpha,e_\beta\rangle D(\ln \rho_1)\rangle\\
=&||h^0||^2+n_1||D(\ln\rho_2)||^2+n_2||D(\ln\rho_1)||^2\\
&\geq n_1||D(\ln\rho_2)||^2+n_2||D(\ln\rho_1)||^2.
\end{align*}
\ref{ca:3}. If $\phi:\n\to\m$ is a $N_i-$totally geodesic immersion, then it  follow from \eqref{eq:doubly.fun.m1} and \eqref{eq:doubly.fun.m2} that $h^0(Z,W)=\langle Z,W\rangle D(\ln \rho_j)$ for $Z, W\in \di_i$. Since $D(\ln\rho_j)$ and $h^0(Z,W)$ are orthogonal, we have $h^0(Z,W)=0$ and $D\ln \rho_j=0$. The first equation implies that $\phi_i$ is totally geodesic and the second equation implies that $\grad (\ln \rho_j)|_{N_j}=\grad(\ln f_j)$.

Conversely, if $\phi_i$ is totally geodesic and $D(\ln \rho_j)=0$, then it follows from  \eqref{eq:doubly.fun.m1} and \eqref{eq:doubly.fun.m2} that $h(Z,W)=0$ for $Z,W\in \di_i$.\\[2mm]
\ref{ca:4}. This follows from equation \eqref{eq:doubly.fun.m12} and statement \ref{ca:3}.
\end{proof}

\begin{thm}
 A doubly warped product immersion
 $$ \phi=(\phi_{1} , \phi_{2}) : \n \to \m $$
 is totally umbilical with mean curvature vector given by $$\textbf{H}^\phi=-(D\ln \rho_{1}+D\ln \rho_{2})$$ if and only if
$ \phi_{1}:N_{1} \to  M_{1} $
and
$ \phi_{2}:N_{2} \to  M_{2} $ are totally umbilical with mean curvature vectors given by
$ -D\ln \rho_{1} $ and $ -D\ln \rho_{2} $, respectively.
\end{thm}
\begin{proof}
Assume that $\phi:\n \to \m$ is a totally umbilical immersion with mean curvature vector given by $\textbf{H}^\phi=-(D\ln \rho_{1}+D\ln \rho_{2})$. Then, we have
\begin{align}\label{eq:totally:umbilical}
h(X,Y)=\langle X, Y\rangle\mathbf{H}^\phi,\quad \forall X,Y\in\di_i
\end{align}
From \eqref{eq:doubly.fun.m1} and \eqref{eq:doubly.fun.m2} we get
\begin{align}\label{eq:totally:umbilical}
h^0(X,Y)=\langle X, Y\rangle\mathbf{H}^i,\quad \forall X,Y\in\di_i, \qquad \textbf{H}^i=-D(\ln \rho_i)
\end{align}
\end{proof}
\begin{thm}\label{th:mini}
Let
$ \phi=(\phi_{1} , \phi_{2}) : \n \to \m $
be a doubly warped product immersion between two doubly warped product manifolds. Then we have
\begin{enumerate}
\item
$ \phi $ is $ N_{i} $-minimal if and only if $ \phi_{i} $ is a minimal immersion and $ D\ln \rho_{j}=0 (i \neq j ) $.
\item
$ \phi $ is a minimal immersion if and only if the mean curvature vectors of $ \phi_{1} $ and $ \phi_{2} $ are given by
$ n_{1}^{-1}n_{2}{f_{1}}^{2}D\ln \rho_{1} $ and $ {n_{1}}n_{2}^{-1}{f_{2}}^{2}D\ln \rho_{2} $, respectively.
\end{enumerate}
\end{thm}

\begin{proof}
Let $e_i\in\di_1, i=1,\ldots,n_1$ and $e_\alpha\in\di_2, \alpha=n_1+1,\ldots,n_1+n_2$  be orthonormal frame fields of $\n$.
If $\phi:\n\to\m$ is $N_i-$minimal, then equations \eqref{eq:doubly.fun.m1} and \eqref{eq:doubly.fun.m2} implies that
\begin{align*}
\textbf{H}_i&=\frac{\text{trace }h_i}{n_i}=\frac{1}{n_i}\sum_{a=1}^{n_i}h(e_a,e_a)\\
&=\frac{1}{n_i}\sum_{a=1}^{n_i}\left( h^0(e_a,e_a)-\langle e_a,e_a\rangle D(\ln\rho_j)\right)\\
0&=\frac{1}{n_i}\sum_{a=1}^{n_i} h^0(e_a,e_a)-D(\ln\rho_j)
\end{align*}
\begin{equation}
\frac{\text{trace }h_i^0}{n_i}-D(\ln\rho_j)=0
\end{equation}
Since $D(\ln\rho_j)$ and $\text{trace }h_i^0$ are orthogonal, we have $\phi_i$ is minimal immersion and $D(\ln\rho_j)=0$.

Conversely, if $\phi_i$  is a minimal immersion and $D(\ln\rho_j)=0$ holds, then it follows from equations \eqref{eq:doubly.fun.m1} and \eqref{eq:doubly.fun.m2} that $\text{trace }h_i=0$. Hence we have statement (1) holds.

If $\phi:\n\to\m$ is a minimal immersion , then $\text{trace }h=0$. Now by applying  equations \eqref{eq:doubly.fun.m1} and \eqref{eq:doubly.fun.m2} we get
\begin{equation}\label{eq:minimal}
\frac{\text{trace }h_1^0}{f_1^2}+\frac{\text{trace }h_2^0}{f_2^2}-n_1D(\ln\rho_2)-n_2D(\ln\rho_1)=0.
\end{equation}
Since $D(\ln\rho_j)$ and $\text{trace }h_i^0$ are tangent to $M_i$ we have $$\dfrac{\text{trace }h_i^0}{f_i^2}-n_jD(\ln\rho_i)=0.$$
\end{proof}
\begin{rem}
If $\phi_i$  is a minimal immersion and $f_i, i=1,2$ is a positive constant, then $\phi_i$ is a minimal immersion.
\end{rem}
\begin{lem}
Let $\phi=(\phi_1,\phi_2):\n\to\m$ be a doubly warped product immersion from a doubly warped product $\n$ into a doubly warped  product $\m$. The shape operator of $\phi$ satisfies
\begin{align}
A_{\eta_1}(X)&=A^0_{\eta_1}(X_1)-\eta_1(\ln\rho_1)X_2\label{eq:shape:warp1}\\
A_{\eta_2}(X)&=A^0_{\eta_2}(X_2)-\eta_2(\ln\rho_2)X_1\label{eq:shape:warp2}\\
D_X\eta_1&=D^0_{X_1}\eta_1+X_2(\ln\rho_2)\eta_1\label{eq:shape:warp3}\\
D_X\eta_2&=D^0_{X_2}\eta_2+X_1(\ln\rho_1)\eta_2\label{eq:shape:warp4}
\end{align}
for $X=X_1+X_2\in\Gamma(TN), X_1\in\di_1, X_2\in\di_2$ and $\eta_1\in\mathcal{L}(M_1)\cap \Gamma(TN)^\bot, \eta_2\in\mathcal{L}(M_2)\cap \Gamma(TN)^\bot$.
 Where $A, D$ and $A^0, D^0$ are the shape operators  and normal connections of a doubly warped product immersion $\phi=(\phi_{1} , \phi_{2}) : \n \to \m $ and the shape operator  normal connection  of the corresponding direct product immersion $\phi=(\phi_{1} , \phi_{2}) : N_1\times N_2 \to M_1\times M_2 $, respectively.
\end{lem}
\begin{proof}
Assume that  $X_1\in\di_1, X_2\in\di_2$ and $\eta_1\in\mathcal{L}(M_1), \eta_2\in\mathcal{L}(M_2)$ are normal to $N=\n$. By applying \eqref{eq:doubly.fun.m1}, \eqref{eq:doubly.fun.m2}, \eqref{eq:doubly.fun.m12} and \eqref{eq:win}   we obtain
\begin{align}
A_{\eta_1}(X_1)&=A^0_{\eta_1}(X_1), \qquad A_{\eta_1}(X_2)=-\eta_1(\ln\rho_1)X_2 \\
D_{X_1}\eta_1&=D^0_{X_1}\eta_1\qquad D_{X_1}\eta_1=X_2(\ln\rho_2)\eta_1\\
A_{\eta_2}(X_2)&=A^0_{\eta_2}(X_2), \qquad A_{\eta_2}(X_1)=-\eta_2(\ln\rho_2)X_1 \\
D_{X_2}\eta_2&=D^0_{X_2}\eta_2\qquad D_{X_1}\eta_2=X_1(\ln\rho_1)\eta_2
\end{align}
Now, our assertion is clear.
\end{proof}
A doubly  warped product manifold $\m$ is called a doubly warped product representation of a real space form $R^m(c)$ of constant sectional curvature $c$ if the doubly warped product $\m$ is an open dense subset of $R^m(c)$.

\begin{thm}
let $\phi=(\phi_1,\phi_2):\n\to\m$ be a doubly warped product immersion from a doubly warped product $\n$ into a doubly warped  product representation  $\m$ of a real space form $R^m(c)$.
 \begin{enumerate}
 \item
The two partial mean curvature vectors $\textbf{H}_1$ and $\textbf{H}_2$ satisfify
 \begin{align}
\langle \textbf{H}_1,\textbf{H}_2\rangle=\frac{\Delta^1 f_1}{n_1f_1}+\frac{\Delta^2 f_2}{n_2f_2}-c
\end{align}
   \item
 The shape operator of $\phi$ satisfies
\begin{align}
A_{\textbf{H}_1}Z=-\frac{\mathcal{H}^{f_2}(Z)}{f_2}+(\frac{\Delta^1 f_1}{n_1f_1}-c)Z, \qquad Z\in\di_2\\
A_{\textbf{H}_2}Z=-\frac{\mathcal{H}^{f_1}(Z)}{f_1}+(\frac{\Delta^2 f_2}{n_2f_2}-c)Z, \qquad Z\in\di_1
\end{align}
 \end{enumerate}
 Where $H^{f_i}$ and $\mathcal{H}^{f_i}$  are the Hessian tensor and the Hessian operator of $f_i$ on $\n$, respectively, i.e. $H^{f_i}(Z,X)=\langle\mathcal{H}^{f_i}(Z),X\rangle$, and $\Delta^i$ is the Laplacian operator of $N_i, i=1,2$.
\end{thm}
\begin{proof}
In \cite{olteanu:generalinequalityfordoublywarped}, Olteanu proved that  for unit vector fields $X\in\di_1$ and $Z\in\di_2$ we have:
\begin{align} \label{eq:k}
K(X\wedge Y)=\frac{1}{f_1}((\co^1_XX)f_1-X^2f_1)+\frac{1}{f_2}((\co^2_ZZ)f_1-Z^2f_2).
\end{align}
If we choose a local orthonormal frame $e_1,\ldots,e_{n_1+n_2}$ in such a way that $e_1, \ldots, e_{n_1}\in\di_1$ and
$e_{n_1+1}, \ldots, e_{n_1+n_2}\in\di_2$ then  \eqref{eq:curvature:cons} and \eqref{eq:k} and \eqref{eq:doubly.fun.m1}, \eqref{eq:doubly.fun.m2}, \eqref{eq:doubly.fun.m12} yields
\begin{align}
K(e_i\wedge e_\alpha)=c+\langle h(e_i,e_i),h(e_\alpha,e_\alpha)\rangle-\langle h(e_i,e_\alpha),h(e_i,e_\alpha)\rangle\\
\frac{\Delta^1 f_1}{f_1}+\frac{n_1}{f_1}(\co^2_{e_\alpha}e_\alpha f_2-e_\alpha f_2)=n_1c+n_1\langle H_1,h(e_\alpha,e_\alpha)\rangle\\
n_2\frac{\Delta^1 f_1}{f_1}+n_1\frac{\Delta^2 f_2}{f_2}=n_1n_2c+n_1n_2\langle \textbf{H}_1,\textbf{H}_2\rangle.
\end{align}
In \eqref{eq:curvature:cons} with choose $X=Z=e_i, i=1\ldots,n_1$ and $Y, W\in\di_2$.
Therefore
\begin{align}
\langle R(e_i,Y)e_i,W\rangle=-c\langle Y,W\rangle-\langle h(e_i,e_i),h(Y,W)\rangle.\label{eq:right:side}
\end{align}
Now Lemma  \eqref{eq:doubly.cu} and direct calculation, gives
\begin{align}
\langle R(e_i,Y)e_i,W\rangle&=\frac{H^{f_1}(e_i,e_i)}{f_1}\langle Y,W\rangle+\frac{1}{f_2}H^{f_2}(Y,W)\nonumber\\
&=\langle\frac{H^{f_1}(e_i,e_i)}{f_1} Y+\frac{1}{f_2}\mathcal{H}^{f_2}(Y),W\rangle.\label{eq:left:side}
\end{align}
 Now from \eqref{eq:left:side}, \eqref{eq:right:side} and \eqref{eq:laplac} we have
\begin{align}
\langle-\frac{\Delta^1 f_1}{f_1} Y+\frac{n_1}{f_2}\mathcal{H}^{f_2}(Y),W\rangle=-cn_1\langle Y,W\rangle-\langle n_1\textbf{H}_1,h(Y,W)\rangle\label{eq:left:side}\\
-\frac{\Delta^1 f_1}{f_1} Y+\frac{n_1}{f_2}\mathcal{H}^{f_2}(Y)=-cn_1Y- n_1A_{\textbf{H}_1}(Y)
\end{align}
Consequently,  $A_{\textbf{H}_1}Y=-\frac{\mathcal{H}^{f_2}(Y)}{f_2}+(\frac{\Delta^1 f_1}{n_1f_1}-c)Y.$
\end{proof}
\begin{cor}
The two partial mean curvature vectors $\textbf{H}_1$ and $\textbf{H}_2$ are perpendicular to each other if and only if  the warping functions $f_1$ and $f_2$ satisfy
  \begin{align}
\frac{\Delta^1 f_1}{n_1f_1}+\frac{\Delta^2 f_2}{n_2f_2}=c.
\end{align}
\end{cor}
\begin{cor}
If the warping functions
$ f_{1} $ and $ f_{2} $
are eigenfunctions of the Laplacian operators
$ \Delta^{1} $ and $ \Delta^{2} $
with eigenvalues $ \dfrac{n_{1}c}{2} $ and $ \dfrac{n_{2}c}{2} $, respectively, then the two partial mean curvature vectors
$ \textbf{H}_{1} $ and $ \textbf{H}_{2} $
are perpendicular to each other.
\end{cor}
\begin{cor}
When $ c=0 $, if the warping functions
$ f_{1} $ and $ f_{2} $
are harmonic functions, then the two partial mean curvature vectors
$ \textbf{H}_{1} $ and $\textbf{H}_{2} $
are perpendicular to each other.
\end{cor}
\section{Inequalities in doubly warped product manifolds }
A. Olteanu \cite{olteanu:generalinequalityfordoublywarped} and   S. Sular and C. {\"O}zg{\"u}r \cite{sular.ozgur:Doublywarpedproduct} studied a geometric inequality and minimal immersion problem. By applying the above mentioned results  we reach to the following propositions.
\begin{thm}[A. Olteanu \cite{olteanu:generalinequalityfordoublywarped}]
Let $\phi$ be an isometric immersion of an $n$-dimensional doubly
warped product $\n$ into an $m$-dimensional arbitrary Riemannian
manifold $\tilde{M}$. Then:
\begin{align}\label{201}
n_{2}\dfrac{\Delta^{1}f_{1}}{f_{1}}+n_{1}\dfrac{\Delta ^{2}f_{2}}{f_{2}}\leq \dfrac{n^{2}}{4}  \|H\|^{2}+n_{1}n_{2}\max \tilde{K},
\end{align}
where $n_i = dim N_i , n = n_1+n_2$, $\Delta^i$ is the Laplacian operator of $N_i, i = 1, 2$.
and $\max \tilde{K}(p)$ denotes the maximum of the sectional curvature function of
$\tilde{M}$ restricted to 2-plane sections of the tangent space $T_pN$ of $N$ at each
point $p$ in $N$. Moreover, the equality case of (\ref{201}) holds if and only if the
following two statements hold
\begin{enumerate}
  \item
 $\phi$ is a mixed totally geodesic immersion satisfying $n_1H_1 = n_2H_2$,
where $H_i, i = 1, 2$, are the partial mean curvature vectors of $N_i$.
\item  At each point $p = (p_1, p_2)\in N$, the sectional curvature function $\tilde{K}$
of $\tilde{M}$ satisfies $\tilde{K}(u, v) = \max \tilde{K}(p)$ for each unit vector $u\in T_{p_1}N_1$
and each unit vector $v\in T_{p_2}N_2$.
\end{enumerate}
\end{thm}

\begin{cor}[A. Olteanu \cite{olteanu:generalinequalityfordoublywarped}]
Let $\phi$ be an isometric immersion of an $n$-dimensional doubly
warped product $\n$  into a Riemannian $m$-manifold $R^m (c)$ of
constant curvature c. Then:
\begin{align}\label{202}
n_{2}\dfrac{\Delta^{1}f_{1}}{f_{1}}+n_{1}\dfrac{\Delta ^{2}f_{2}}{f_{2}}\leq \dfrac{n^{2}}{4}  \|H\|^{2}+n_{1}n_{2}c,
\end{align}
where $n_i = dim N_i , n = n_1+n_2$, $\Delta^i$ is the Laplacian operator of $N_i, i = 1, 2$.
Moreover, the equality case of \eqref{202} holds if and only if $\phi$ is a mixed totally
geodesic immersion satisfying $n_1\textbf{H}_1 = n_2\textbf{H}_2$, where $\textbf{H}_i, i = 1, 2$, are the
partial mean curvature vectors of $N_i$.
\end{cor}
\begin{cor}
Let
$ \n\to\mathit{R^{m}}(c)$
be an isometric immersion of a doubly warped product
$ \n $
into a real space form
$ \mathit{R^{m}}(c) $
of constant curvature
$ c. $, and the warping functions
$ f_{1} $ and $ f_{2} $ are eigenfunctions of the Laplacian on
$ N_{1} $ and $ N_{2} $ with eigenvalues
$ \dfrac{n_{1}c}{2} $ and $ \dfrac{n_{2}c}{2} $, respectively,
Then, the equality holds in \\
\begin{align}
n_{2}\dfrac{\Delta^{1}f_{1}}{f_{1}}+n_{1}\dfrac{\Delta ^{2}f_{2}}{f_{2}}\leq \dfrac{n^{2}}{4}  \|H\|^{2}+n_{1}n_{2}c,
\end{align}
if and only if
$ \phi $
is a minimal immersion.
\end{cor}

\begin{prop}\label{pro:4.3}
Let
$ \n $
be a doubly warped product of two Riemannian manifolds whose warping functions
$ f_{1} $ and $ f_{2} $
are harmonic functions. Then
\begin{enumerate}
\item
$ \n $
admits no isometric minimal immersion into any Riemannian manifold of negative curvature;
\item
Every isometric minimal immersion from
$ \n$
into a Euclidean space  is a mixed totally geodesic immersion.
\end{enumerate}
\end{prop}
\begin{proof}
Assume that
$ \phi:\n\to \tilde{M}$
is an isometric minimal immersion of a doubly warped product
$\n$
into a Riemannian manifold $\tilde{M}$.
 If
$ f_{1} $ and $ f_{2} $
are harmonic functions on
$ N_{1} $ and $ N_{2} $, respectively, then inequality \eqref{201}
implies
$\max \tilde{K}\geq0$ on
the doubly warped product
$ \n $. This shows that $\n$
does not admit any isometric minimal immersion into a any Riemannian manifold of negative curvature.

When
$\tilde{M}$ is a Euclidean space($ c=0 $) the minimality of $\n$ and the harmonicity  of $f_1$ and $f_2$ imply that
the equality in (\ref{202}) holds identically. Thus, the immersion is mixed totally geodesic.
\end{proof}
\begin{cor}
Let
$ \n $
be a doubly warped product of two Riemannian manifolds whose warping functions
$ f_{1} $ and $ f_{2} $
are harmonic functions and one of $N_i, i=1,2$ is compact. Then
every isometric minimal immersion from
$ \n$
into a Euclidean space is a warped product immersion.

\end{cor}

\begin{proof}
Let $\phi:\n\to\mathbb{E}$ be an  isometric minimal immersion and $N_2$ be compact. Since $f_2$ is harmonic, by applying lemma \ref{lemma:hopf1} and the compactness of $N_2$, we know that $f_2$ is a positive constant. Therefore, the doubly warped product $\n$ can be expressed as a warped product $\tilde{N}_1\times_{f_1} N_2$ where $\tilde{N}_1=N_1$, equipped with the metric $f_2^2g_1$ which is homothetic to the original metric $g_1$ on $N_1$. Now,  Theorem 5.2  in  \cite{chen:on.isometric.minimal.immersions} implies that $\phi$ is a warped product immersion.
\end{proof}

\begin{prop}
If
$ f_{1} $ and $ f_{2} $ are eigenfunctions of the Laplacian on
$ N_{1} $ and $ N_{2} $ with eigenvalues
$ n_{1}\lambda $ and $ n_{2}\lambda, \lambda>0, $(or with eigenvalues $\lambda$)   respectively, then
$ \n $
does not admit an isometric minimal immersion into 
any Riemannian manifold of non-positive curvature.
\end{prop}
\begin{proof}
The Inequality (\ref{201}) implies that
$ n_1n_2\max\tilde{K}\geq \lambda>0. $
Hence, the space
$ \tilde{M}$
cannot be non-positive curvature.
\end{proof}

\begin{prop}
If $N_1$  is a compact Riemannian manifold and $f_2$ is a harmonic function on $N_2$, then
\begin{enumerate}
  \item
Every doubly warped product $\n$  does not admit an isometric minimal immersion into 
any Riemannian manifold of negative curvature;
\item
Every doubly warped product $\n$  does not admit an isometric minimal immersion into a Euclidean space.
\end{enumerate}
\end{prop}
\begin{proof}
Assume $N_1$ is compact and $\phi:\n\to M$ is an isometric minimal immersion of $\n$ into a non-positive curvature Riemannian manifold $M$. From harmonicity  of $f_2$ and inequality \ref{201} we have
$$\frac{\Delta^1 f_1}{f_1}\leq n_1n_2\max\tilde{K}\leq0.$$
Since the warping function $f_1$ is positive, we find $\Delta^1 f_1\leq0$. Hence. it follows from Hopf's lemma \ref{lemma:hopf2} that $f$ is a positive constant. We have $\max \tilde{K}=0$, which implies (1).

In $M=\mathbb{E}^m$, the similar proof as case (1), $f_1$ is a positive constant. Hence the equality case of \eqref{202} holds and  $\phi$ is mixed totally geodesic.
Since $f$ is a positive constant, then  doubly warped product $\n$  is a warped product of the Riemannian manifold $(N_1,g_1)$ and the Riemannian manifold $\tilde{N}_2=(N_2,f_1^2g_2)$, that is ${N_{1}}_{f_2}\!\!\times\tilde{N}_2$.
By applying a result of N\"{o}lker in \cite{nolker:Isometric.immersions.of.warped.products}, $\phi$  is warped product immersion, say
$$\phi=(\phi_1,\phi_2):{N_{1}}_{f_2}\!\!\times\tilde{N}_2\to M=\mathbb{E}^{m_1}\times \mathbb{E}^{m_2}$$
By Theorem \ref{th:mini}, $\phi$ is minimal, $\phi_1:N_1\to\mathbb{E}^{m_1}$ is  minimal since $\phi$ is minimal. This is impossible by Lemma \ref{lemma:mini} since $N_1$ is compact.

\end{proof}



\bigskip

\bibliographystyle{siam}
\def\polhk#1{\setbox0=\hbox{#1}{\ooalign{\hidewidth
  \lower1.5ex\hbox{`}\hidewidth\crcr\unhbox0}}}

\end{document}